\documentclass[11pt]{amsart}
\usepackage{amsmath,amsfonts,amssymb,latexsym}
\def\R{{\mathbb R}}
\def\N{{\mathbb N}}

\newcommand{\ssi}{if and only if}

\newcommand{\sep}{\mbox{-}}
\newcommand{\ld}{\lambda}
\newcommand{\var}{\begin{array}[t]{c}+\\[-7pt]  {\scriptstyle v}\end{array}}
\newcommand{\tos}{\rightrightarrows}

\newcommand{\Dom}{{\rm Dom} \kern.15em}

\newcommand{\ps}{\smallbreak}

\newcommand{\la}{\langle}
\newcommand{\ra}{\rangle}

\newtheorem{theo}{Theorem}[section]

\newtheorem{lemma}[theo]{Lemma}
\newtheorem{prop}[theo]{Proposition}
\newtheorem{cor}[theo]{Corollary}

\newcounter{ex}
\renewcommand{\theex}{\arabic{ex}}

\begin{document}
\thispagestyle{empty}
\begin{abstract}
We show that the lower limit of a sequence of  maximal monotone operators
on a reflexive Banach space is a representable monotone operator. As a consequence, we obtain that the variational
sum of maximal monotone operators and the variational composition of a maximal monotone operator
with a linear continuous operator are both representable monotone operators.
\end{abstract}
\title[Limits of sequences of maximal monotone operators]%
{Representable monotone operators and limits of sequences of maximal monotone operators}
\author{Yboon Garc\'{\i}a}
\address{IMCA and Facultad de Ciencias de la Universidad Nacional de Ingenier\'{i}a, Lima, Peru}
\email{yboon@imca.edu.pe}
\author{Marc Lassonde}
\address{LAMIA, Universit\'e des Antilles et de la Guyane,
  97159 Pointe \`a Pitre, France}
\email{marc.lassonde@univ-ag.fr}

\date{December 17, 2010}



\maketitle
\section{Introduction}
\setcounter{equation}{0}
In recent years, the utilization of Fitzpatrick functions in the study of monotone operators
gave rise to a new class of monotone operators, the so-called \textit{representable} operators.
These are operators $T: X\tos X^*$ from a Banach space $X$ into its dual $X^*$
whose graphs can be described via a convex lower semicontinuous function $f$ defined on $X\times X^*$,
namely:
$T=\{(x,x^*) \in X\times X^* : f(x,x^*)= \langle x^*, x \rangle\}$.
Maximal monotone operators are representable, but not all representable operators
are maximal monotone. However, it turns out that representable operators
possess properties similar to those
of maximal monotone operators, and that representable operators with full-space domain
are actually maximal monotone (see Section 2). Thus, in situations where maximal monotonicity
is lacking or not known, it is interesting to know whether representability is present.
Three such situations are studied in this paper: the lower limit of sequences of maximal monotone
operators (Section 3), the variational sum of two maximal monotone operators (Section 4)
and the variational composition of a maximal monotone operator with a linear continuous map (Section 5).
\ps
Let $T_n:X\tos X^*$ be a sequence of maximal monotone operators
between a reflexive strictly convex Banach space $X$ and its strictly convex dual $X^*$.
It is known that the strong (graph-)lower limit of such a sequence, denoted $\liminf\,T_n$,
is monotone but not maximal monotone in general.
Here we show that it is however representable (Theorem \ref{theo1}).
On the other hand, it is known that in finite-dimensional spaces,
the so-called Painlev\'e-Kuratowski (graph-)limit of the $T_n$'s is indeed maximal monotone
(Attouch's theorem \cite{Att79,ABT94}).
We extend this theorem to the case of arbitrary reflexive strictly convex Banach spaces,
replacing, as usual, Painlev\'e-Kuratowski convergence by Mosco convergence (Theorem \ref{theo2}).
\ps
Next, we consider the \textit{variational sum} of two maximal monotone operators
$T_1,T_2$. This concept was introduced by Attouch-Baillon-Th\'era \cite{ABT94} in Hilbert spaces
as a substitute for the usual (Minkowski) sum $T_1+T_2$ which in general does not yield a
maximal monotone operator.
It is given by
$$
T_1 \var T_2 := \bigcap_\mathcal{I}\liminf_n\,(T_{1, \lambda_n}+ T_{2,\mu_n}),
$$
where
$
\mathcal{I}=\Bigl\{ \{(\lambda_n, \mu_n)\}\subset \R^2: \lambda_n,\, \mu_n \geq 0, \,
\lambda_n+ \mu_n >0, \lambda_n, \mu_n \to 0 \Bigr\},
$
and where $T_{1, \lambda_n}$ and $T_{2,\mu_n}$ denote the Yosida regularizations of $T_1$ and $T_2$
respectively.
The operators $T_n=T_{1, \lambda_n}+ T_{2,\mu_n}$ are maximal monotone, so
as a consequence of our previous result on the lower limit of sequences of maximal monotone operators,
we derive that the variational sum is actually a representable extension of the usual sum $T_1+T_2$.
\ps
Analog results for \textit{variational composition} are also obtained in the last section.

\section{Maximal monotone and representable operators}
\setcounter{equation}{0}
Set-valued mappings $T:X\tos Y$ between sets $X$ and $Y$
are identified with their graphs $T\subset X\times Y$,
so $x^*\in Tx$ is equally written as $(x,x^*)\in T$.
The \textit{values} of $T:X\tos Y$ are the subsets $Tx\subset Y$
for $x\in X$, the \textit{inverse} of $T$ is the mapping
$T^{-1}: Y \tos X$ defined by
$
T^{-1}y= \bigl\{x \in X: y \in Tx\bigr\},
$
and the domain of $T$ is the projection of (the graph of) $T$ onto $X$, that is,
$\Dom T=\{ x\in X : Tx\ne \emptyset\,\}$.

\ps
In what follows, $X$ denotes a Banach space, $X^*$ its continuous dual, $B_{X^*}$ the unit ball in $X^*$,
and $X\times X^*$ is equipped with the strong$\times$weak-star ($s\times w^*$) topology.
Recall that a set-valued operator $T:X\tos X^*$, or a subset $T\subset X\times X^*$,
is said to be
\ps
$\bullet$ {\it monotone} provided
$
\langle y^*-x^*,y-x\rangle \geq 0,\  \forall (x,x^*),(y,y^*)\in T,
$
\ps
$\bullet$ {\em maximal monotone} provided it is monotone and
maximal (under set inclusion) in the family of all monotone sets contained in $X\times X^*$,
\ps
$\bullet$ {\em representable} provided there is a lower semicontinuous
convex function $f: X\times X^* \to \mathbb{R}\cup\{+\infty\}$ such that
$$
\left\lbrace\begin{array}{l}
f(x,x^*)\ge \langle x^*, x \rangle,\quad \forall \,(x,x^*) \in X \times X^*,\label{equ11}\smallskip\\
f(x,x^*)= \langle x^*, x \rangle \Leftrightarrow (x,x^*) \in T.\label{equ12}
\end{array}\right.
$$
Every representable operator is indeed
monotone (see, e.g., Penot-Z{\u{a}}linescu \cite{PZ05}).
\ps
Using the notations of Mart\'\i nez-Legaz-Svaiter \cite{M-LS05},
$$\mathcal{F}=
\{ f: X\times X^* \to \mathbb{R}\cup\{+\infty\}
\mbox{ lower semicontinuous convex} : f(x,x^*)\ge \langle x^*, x \rangle\ \forall \,(x,x^*)\},
$$
and, given $f\in\mathcal{F}$, 
$$L(f)=\{(x,x^*) \in X\times X^* : f(x,x^*)= \langle x^*, x \rangle\},$$
we can write more synthetically:
\begin{equation}\label{defrep}
T \mbox{ representable }\Longleftrightarrow \exists f\in \mathcal{F} \,:\, T=L(f).
\end{equation}

\medbreak
For a nonempty $T:X\tos X^*$, consider $\varphi_T,\varphi_T^*:X\times X^* \to \mathbb{R}\cup\{+\infty\}$ given by
$$
\left\lbrace\begin{array}{l}
\varphi_T(x,x^*)  = \sup\,\{\,\langle y^*, x\rangle -\langle y^*,y\rangle+\langle x^*, y\rangle : (y,y^*)\in T\,\},
\medskip\\
\varphi_T^*(x,x^*)= \sup\,\{\,\langle (y,y^*), (x,x^*) \rangle-\varphi_T(y,y^*) : (y,y^*) \in X\times X^*\,\}.
\end{array}\right.
$$
Obviously, $\varphi_T$ and $\varphi_T^*$ are convex and lower semicontinuous in $X\times X^*$ supplied with
the strong$\times$weak-star topology. Both functions were considered by Fitzpatrick \cite{Fit88}.
Since then, they have been recognized to be quite useful
in the study of monotone operators
(see, e.g., \cite{BWY09,BWY10,Bor06,M-LS05,Pen04,Sim08,SZ05,Sva03,Voi06,VZ10}).
The following proposition shows examples of 
classifications of monotone operators using these functions.
\begin{prop} [see, e.g., \cite{M-LS05,Pen04,Voi06}]\label{propmono}
Let $X$ be a Banach space and let $T:X\tos X^*$ with $\Dom T\ne\emptyset$. Then:
\begin{enumerate}
\item $T$ monotone $\Longleftrightarrow$ $\varphi_T^*\in \mathcal{F}$ and $T\subset L(\varphi_T^*)$.
\item $T$ representable $\Longleftrightarrow$ $\varphi_T^*\in \mathcal{F}$ and $T=L(\varphi_T^*)$.
\item $T$ maximal monotone $\Longleftrightarrow$ $\varphi_T\in \mathcal{F}$ and $T=L(\varphi_T^*)$.
\end{enumerate}
\end{prop}

\noindent
{\it Example.}
The operator $T=\{(0,0)\}\subset \R\times\R$ (considered in Fitzpatrick \cite{Fit88})
is representable but not maximal monotone. Easy computation shows that
$\varphi_T(x,x^*)=0$ for every $(x,x^*)$, while
$\varphi_T^*(x,x^*)=\delta_T(x,x^*)$ for every $(x,x^*)$, where
$\delta_T$, the \textit{indicator function} of (the graph of) $T$, is equal to $0$ on $T$,
to $+\infty$ outside.
Hence $\varphi_T\not\in\mathcal{F}$, $\varphi_T^*\in\mathcal{F}$, and $T=L(\varphi_T^*)$.
Any linear map is a maximal monotone extension of $T$, and $T$ is equal to the intersection of all
its maximal monotone extensions.
\ps
The next proposition gives elementary properties of representable operators.

\begin{prop}\label{improp}
Let $X$ be a Banach space and let $T:X\tos X^*$ be representable. Then:
\begin{enumerate}
\item  $T$ and $T^{-1}$ have convex values.
\item  For any $n\in\N$, $T\cap (X\times nB_{X^*})$ is $s\times w^*$-closed in $X\times X^*$.
 \end{enumerate}
\end{prop}
\begin{proof}
(1) is well known and obvious. For (2), consider a net $\{(x_\nu, x^*_\nu)\}\subset T$, with
$\{x^*_\nu\}\subset nB_{X^*}$, which $s\times w^*$-converges to $(x,x^*)$.
Clearly, $x^*$ belongs to $nB_{X^*}$. It remains to show that $(x,x^*)\in T$.
By Proposition \ref{propmono}\,(2),
$T=L(\varphi_T^*)$, so
$
\varphi_T^*(x_\nu, x^*_\nu)=\la x^*_\nu,x_\nu\ra.
$
Since $\{x^*_\nu\}$ is bounded, we have $\la x^*_\nu,x_\nu\ra\to \la x^*,x\ra$, hence,
from the $s\times w^*$-lower semicontinuity of $\varphi_T^*$ we derive that
$\varphi_T^*(x, x^*)\le \la x^*,x \ra$. Since $\varphi_T^*\in \mathcal{F}$,
the reverse inequality also holds, so $(x, x^*)\in L(\varphi_T^*)=T$.
\end{proof}

A monotone operator $T:X\tos X^*$ is said to
be \textit{maximal monotone in $\Omega\subset X$} provided the monotone set
$
T\cap (\Omega\times X^*)
$
is maximal (under set inclusion) in the family of all monotone sets contained in
$\Omega\times X^*$.
A set-valued mapping $T:Z\tos Y$ between topological
spaces $Z$ and $Y$ is said to be {\em upper semicontinuous}
at $z\in Z$ provided for any open $V$ containing
$Tz$ there is an open neighborhood $U$ of $z$
such that $T(U) \subset V$. When $Y$ is a linear topological space
with topology $\tau$, an upper semicontinuous $T:Z\tos Y$
with nonempty $\tau$-compact convex values is called {\em $\tau$-cusco},
and a $\tau$-cusco $T:Z\tos Y$ is called {\em minimal $\tau$-cusco}
if its graph does not contain the graph of any other $\tau$-cusco
between $Z$ and $Y$.
\ps
The next theorem exhibits 
a case where representable and maximal monotone operators coincide.

\begin{theo} \label{usc}
Let $X$ be a Banach space and let $T:X\tos X^*$ be representable
with $\Omega={\rm int\,}\Dom T\ne \emptyset$.
Then, $T$ is maximal monotone in $\Omega$.
In particular, $T$ is minimal $w^*$-cusco in $\Omega$.
\end{theo}
\begin{proof}
We claim that $T$ is strong-to-weak$^*$ upper semicontinuous
at every point in $\Omega$.
Indeed, suppose $T$ is not strong-to-weak$^*$ upper semicontinuous
at some $x\in\Omega$.
Then, there exists an $w^*$-open set $V$ containing $Tx$ and sequences
$\{x_n\}\subset X$ and $\{x^*_n\}\subset X^*$ such that
\begin{equation}\label{canto}
x_n\to x\quad\mbox{and}\quad \forall n\in\N,\,x^*_n\in Tx_n\setminus V.
\end{equation}
Since $T$ is locally bounded at $x$ (see, e.g., Phelps \cite[Theorem 2.28]{Phe93}),
the sequence $\{x^*_n\}$ is bounded. Hence, it admits a bounded subnet $\{x^*_\alpha\}$
$w^*$-converging to a certain $x^*$. Thus, 
for some $n\in\N$,
the net $\{(x_\alpha,x^*_\alpha)\}$ is contained in $T\cap (X\times nB_{X^*})$ and
$s\times w^*$-converges to $(x,x^*)$. From Proposition \ref{improp},
we conclude that $(x,x^*)\in T$.
On the other hand, (\ref{canto}) implies that $x^*_\alpha\not\in V$ for every $\alpha$, hence
$x^*\not\in V$. This is a contradiction because $x^*\in Tx\subset V$. 

Thus, $T:\Omega\tos  X^*$ is monotone and strong-to-weak$^*$ upper semicontinuous
with nonempty $w^*$-closed convex values (by Proposition \ref{improp}) on the open set
$\Omega$. Therefore, $T$ is maximal monotone and minimal $w^*$-cusco in $\Omega$
(see, e.g., \cite[Lemma 7.7 and Theorem 7.9]{Phe93}).
\end{proof}

Let $T$ be a non-empty monotone operator.
By Proposition \ref{propmono}\,(1) and (\ref{defrep}),
$L(\varphi_T^*)$ is a representable extension of $T$.
Actually, $L(\varphi_T^*)$ is the smallest representable extension of $T$, that is,
$L(\varphi_T^*)$ is equal to the intersection of all the representable extensions of $T$.
In finite-dimensional spaces, a more precise result holds:
$L(\varphi_T^*)$ is equal to the intersection of all the maximal monotone extensions of $T$
(see Mart\'\i nez-Legaz-Svaiter \cite{M-LS05} for the details).
\ps
In the sequel, we shall use the following notation for a non-empty monotone operator $T$: we let
$
T^0=\{\,(x,x^*)\in X\times X^* : \langle y^*-x^*,y-x\rangle \geq 0,\  \forall (y,y^*)\in T\,\}
$
and
$
\mathcal{M}(T)=\{ S\subset X\times X^* : S \mbox{ maximal monotone},\, T\subset S\}.
$
Then
$T^0=\bigcup \{ S : S\in \mathcal{M}(T)\}$
and
$T^{00}=\bigcap \{ S : S\in \mathcal{M}(T)\}$, so that $T\subset T^{00} \subset T^0$.
We have:
\begin{enumerate}
\item  $T=T^0$ \ssi\ $T$ is maximal monotone,
\item  $T=T^{00}$ \ssi\ $T$ is the intersection of all its maximal monotone extensions,
\item  $T^0=T^{00}$ \ssi\ $T$ has a unique maximal monotone extension.
\end{enumerate}

\section{Limits of sequences of maximal monotone operators}
\setcounter{equation}{0}
Let $\{T_n\} \subset X \times Y$ be a sequence of operators between topological spaces
$(X, \tau )$ and $(Y, \beta)$.
The \textit{sequential lower limit} of $\{T_n\}$,
w.r.t.\ the product topology $\tau \times \beta$, is the operator
$$
\tau \times \beta\mbox{-}\liminf T_n=
\{\tau \times \beta\mbox{-}\lim_n (x_n,y_n): (x_n,y_n) \in T_n,\, \mbox{for all } n \in \N \},
$$
while its \textit{sequential upper limit} is the operator
$$
\tau \times \beta\mbox{-}\limsup T_n=
\{\tau \times \beta\mbox{-}\lim_k (x_{n_k},y_{n_k}): (x_{n_k},y_{n_k}) \in T_{n_k},
\mbox{for an infinite } \{{n_k}\}\subset\N \}.
$$

Let now $\{T_n\}\subset X \times X^*$ be a sequence from a reflexive Banach space to its dual.
Denote by $s$ and $w$ respectively the strong and weak topologies in $X$ and $X^*$.
Then, $\{T_n\}$ is said to \textit{converge in the sense of Painlev\'e-Kuratowski}
to the operator $T$, written $T= \mbox{PK-}\lim T_n$, if
$$
s \times s\mbox{-}\limsup T_n \subset T\subset s \times s\mbox{-}\liminf T_n,
$$
while $\{T_n\}$ is said to \textit{converge in the sense of Mosco} to $T$,
written $T= \mbox{M-}\lim T_n$, if
$$
w \times w\mbox{-}\limsup T_n \subset T\subset s \times s\mbox{-}\liminf T_n.
$$
In the sequel, we simply write $\liminf T_n$
instead of $s \times s\mbox{-}\liminf T_n$.
\ps From now on, we assume that $X$ is a reflexive Banach space, with a strictly convex
norm in $X$ and a strictly convex dual norm in $X^*$
(always possible thanks to Asplund \cite{As1}).
With such norms, the {\it duality mapping} $J:X\tos X^*$ given by
$$
Jx:= \bigl\{x^*\in X^*: \langle x^*, x \rangle = \|x\|^2 = \|x^*\|^2\bigr\}
$$
is single-valued, bijective and maximal monotone,
and it is well known (see, e.g., Rockafellar \cite{Ro1}) that
a monotone operator $T:X\tos X^*$ is maximal monotone \ssi\
the operator $J+T$ is surjective. Moreover, in that case, the operator
$(J+T)^{-1}$ is single-valued on $X^*$.
According to this result, to any maximal monotone $T:X\tos X^*$ and $\ld>0$,
we associate its \textit{resolvent} $J_\lambda^T: X\to X$ which to
each $x\in X$ assigns the unique solution $x_\lambda\in\Dom T$
of the inclusion $0\in J(x_\lambda-x) + \lambda Tx_\lambda$, and
its \textit{Yosida regularization}
$T_\lambda : X\tos X^*$ given by $T_\lambda x = J(x-x_\lambda)/\lambda$. Then,
$T_\lambda$ is a single-valued maximal monotone operator of all of $X$ into $X^*$,
and satisfies $T_\lambda x\in Tx_\lambda$.
\ps
Let now $\{T_n: X \tos X^*\}$ be a sequence of maximal monotone operators.
For $(x,x^*) \in X \times X^*$ and $n\in \N$,
we consider the unique solution $x_n=J_{T_n}(x,x^*)$ of the inclusion
\begin{equation} \label{reussi1}
 x^* \in J(x_n-x) + T_n(x_n).
\end{equation}
\begin{lemma}\label{lemma1}
Let $X$ be reflexive strictly convex with a strictly convex dual $X^*$ and
let $\{T_n: X \tos X^*\}$ be a sequence of maximal monotone operators.
Set $T=\liminf T_n$ and assume $\Dom T\ne \emptyset$.
Then, for any $(x,x^*) \in X \times X^*$,
the sequence $\{x_n\}$ of solutions of (\ref{reussi1}) is bounded,
and for any subsequence $\{x_{n_k}\}$ with $x_{n_k} \rightharpoonup \overline{x}$ and
$J(x_{n_k}-x)\rightharpoonup \eta^*$, we have
\begin{equation}\label{reussi2}
\langle x^*- \eta^* - y^*, \overline{x} - y \rangle + \langle \eta^*, \overline{x} - x\rangle
\geq \limsup \|x_{n_k} - x\|^2, \quad \forall (y,y^*)\in T.
\end{equation}
In particular, $(\overline{x},x^*- \eta^*)\in T^0$.
\end{lemma}

\begin{proof}
Let $(y,y^*) \in T$. By definition of $T$, there exists a sequence $\{(y_n, y^*_n)\}$
such that $(y_n, y^*_n) \to (y,y^*)$ and $(y_n, y^*_n)\in T_n$ for each $n \in \N$.
Using  the monotonicity of $T_n$  and (\ref{reussi1}) we get
$$
\bigl\langle x^*-J(x_n -x)  - y_n^*, x_n - y_n\bigr\rangle \geq  0,
$$
hence
\begin{equation}\label{reussi4}
\langle x^*  - y_n^*, x_n - y_n\rangle \geq  \|x_n -x\|^2 + \langle J(x_n -x) , x - y_n\rangle.
\end{equation}
Since the sequences $\{y_n\}$ and $\{y^*_n\}$ are bounded, from (\ref{reussi4}) we deduce that
$\{x_n\}$ is bounded.
\ps
Let $\{x_{n_k}\}$ be a subsequence such that $x_{n_k} \rightharpoonup \overline{x}$
and $J(x_{n_k}-x)\rightharpoonup \eta^*$.
Taking the limit in (\ref{reussi4}) we obtain
$$
\langle x^*  - y^*, \overline{x} - y\rangle  \geq  \displaystyle\limsup\|x_{n_k} -x\|^2 + \langle \eta^* , x - y\rangle,
$$
which clearly gives (\ref{reussi2}).
\ps
On the other hand, the monotonicity of $J$ implies
$$
\bigl\langle J(x_{n_k} - x) - J(\overline{x} - x), (x_{n_k}-x) - (\overline{x}-x)\bigr\rangle \geq 0,
$$
that is,
$$
\|x_{n_k} - x\|^2  \geq  \bigl\langle J(x_{n_k} - x), \overline{x}-x \bigr\rangle -
\bigl\langle J(\overline{x} - x), x_n -\overline{x}\bigr\rangle,
$$
so passing to the limit we get
$$
\liminf\|x_{n_k} - x\|^2  \geq  \langle \eta^*, \overline{x}-x \rangle.
$$
Combining  this inequality with (\ref{reussi2}) we derive that
$$
\langle x^*- \eta^* - y^*, \overline{x} - y \rangle\ge 0, \quad \forall (y,y^*)\in T,
$$
proving that $(\overline{x},x^*- \eta^*)\in T^0$.
\end{proof}

The following proposition provides a convenient description of the set
$\liminf T_n$ in terms
of limits of sequences of solutions of (\ref{reussi1}).
\begin{prop}\label{prop1}
Let $X$ be reflexive strictly convex with a strictly convex dual $X^*$.
Let $\{T_n: X \tos X^*\}$ be a sequence of maximal monotone operators.
Then, $(x,x^*)\in \liminf T_n$ \ssi\ $x=\lim x_n$,
where $\{x_n\}$ is the sequence of solutions of
equation (\ref{reussi1}).
\end{prop}
\begin{proof}
Let $(x,x^*)\in \liminf T_n$. Consider
the sequence $\{x_n\}$ of solutions of (\ref{reussi1}).
By Lemma \ref{lemma1}, any subsequence of $\{x_n\}$ admits a converging
subsequence $\{x_{n_k}\}$ which satisfies (\ref{reussi2}).
Using (\ref{reussi2}) with $(y, y^*) = (x, x^* )$, we get
$
0\geq \limsup \|x_{n_k} - x\|^2,
$
that is, $x_{n_k} \to x$. We conclude that the whole sequence $\{x_{n}\}$ converges to $x$.
\ps
Conversely, by (\ref{reussi1}), $(x_n, x^*-J(x_n-x))\in T_n$ for every $n\in \N$.
If $x_n\to x$, then $x^*-J(x_n-x)\to x^*$, so $(x,x^*)\in \liminf T_n$
by definition of $\liminf T_n$.
\end{proof}

We now turn to our main result in this section.
\begin{theo}\label{theo1}
Let $X$ be reflexive strictly convex with a strictly convex dual $X^*$ and
let $\{T_n: X \tos X^*\}$ be a sequence of maximal monotone operators.
\ps {\rm(1)} $T=\liminf T_n$ is representable;
\ps {\rm(2)} If the duality mapping $J$ is weakly sequentially continuous
(as when $X$ is a Hilbert space or $\ell_p$, $1<p<\infty$) and
$T=w\times w\sep\liminf T_n$, then $T$
is the intersection of all its maximal monotone extensions.
\end{theo}

\begin{proof}
(1) Without loss of generality, we may assume that $T$ is nonempty
(the empty set is representable).
It is well known and easily seen that $T$ is monotone, so, by Proposition \ref{propmono},
proving that $T$ is representable amounts to proving that
$L(\varphi_T^*)\subset T$.
According to Lemma \ref{lemma1}, for any $(x,x^*) \in X \times X^*$,
the sequence $\{x_n\}$ of solutions of (\ref{reussi1}) is bounded
and any subsequence $\{x_{n_k}\}$ such that $x_{n_k} \rightharpoonup \overline{x}$ and
$J(x_{n_k}-x)\rightharpoonup \eta^*$ satisfies (\ref{reussi2}).
By definition of $\varphi_T$, (\ref{reussi2}) can be rewritten as
\begin{equation}\label{reussi5}
\langle x^*- \eta^*, \overline{x} \rangle + \langle \eta^*, \overline{x} - x\rangle
\geq \limsup \|x_{n_k} - x\|^2+ \varphi_T(\overline{x}, x^*- \eta^*).
\end{equation}
Since
\begin{eqnarray*}
\varphi_T^*(x,x^*)&= &\sup\,\{\,\langle (y,y^*), (x,x^*) \rangle-\varphi_T(y,y^*) : (y,y^*) \in X\times X^*\,\}\\
    &\ge &\la (\overline{x},x^*- \eta^*), (x,x^*) \rangle-\varphi_T(\overline{x}, x^*- \eta^*)\\
    &= &\la x^*,\overline{x}\ra + \la x^*- \eta^*,x\ra -\varphi_T(\overline{x}, x^*- \eta^*),
\end{eqnarray*}
we derive from (\ref{reussi5}) that
\begin{equation}\label{reussi6b}
\varphi_T^*(x, x^*) \geq \limsup \|x_{n_k} - x\|^2+ \langle x^* , x\rangle.
\end{equation}
Now, let $(x,x^*)\in L(\varphi_T^*)$. Then (\ref{reussi6b}) reduces to
$$
0 \geq \limsup \|x_{n_k} - x\|^2,
$$
showing that  $x_{n_k} \to x$. It follows that $x_n \to x$,
since this subsequence was taken arbitrarily. From Proposition \ref{prop1}
we conclude that $(x,x^*) \in T$, as required.
\medbreak
(2) Again, without loss of generality, we may assume that $T$ is nonempty
(the empty set is equal to the intersection of all the maximal monotone sets).
Assuming $J$ weakly sequentially continuous and $T=w \times w\mbox{-}\liminf T_n$,
we show that $T^{00}=T$.
By Lemma \ref{lemma1}, for any  $(x,x^*)\in X\times X^*$, the sequence $\{x_n\}$ of solutions of (\ref{reussi1})
is bounded and if $\{x_{n_k}\}$ is any subsequence such that $x_{n_k} \rightharpoonup \overline{x}$ and
$J(x_{n_k}-x)\rightharpoonup \eta^*$, then $(\overline{x},x^*-\eta^*)\in T^0$.
In fact $\eta^*=J(\overline{x}-x)$ since $J$ is
weakly sequentially continuous,
so $(\overline{x},x^*-J(\overline{x}-x))\in T^0$.
Now, let $(x,x^*)\in T^{00}$. Then
$$
\bigl\langle x^* - J(\overline{x}-x)-x^*, \overline{x}-x\bigr\rangle\ge 0,
$$
that is,
$
-\|\overline{x}-x\|^2\ge 0.
$
Hence, $\overline{x}=x$. We derive that $x_n\rightharpoonup x$ and
$x^* - J(x_n-x)\rightharpoonup x^*$. Since
$$
\bigl(x_n, x^* - J(x_n-x)\bigr) \in  T_{n},\quad\forall n\in \N,
$$
we have $(x, x^*) \in w \times w\mbox{-}\liminf T_n$,
therefore, by assumption, $(x, x^*) \in T$.
This shows that $T^{00}\subset T$, which was to be proved.
\end{proof}

\begin{cor}
Let $X$ be a finite dimensional space and
let $\{T_n: X \tos X\}$ be a sequence of maximal monotone operators.
Then, $\liminf T_n$ is the intersection of all its maximal monotone extensions.
\end{cor}
\begin{proof}
In finite dimensional spaces, strong and weak topologies coincide.
The result therefore follows directly from Assertion (2) in the theorem.
\end{proof}

\noindent{\it Example.}
In $\R\times\R$, let $T_n=\{0\}\times \R$ for even $n$, $T_n=\R \times \{0\}$ for odd $n$.
Then:
\begin{enumerate}
\item Each $T_n$ is maximal monotone,
\item $\liminf T_n = \{(0,0)\}$ is equal to the intersection of all its maximal monotone extensions,
but is not maximal monotone,
\item $\limsup T_n = (\{0\}\times \R) \cup (\R \times \{0\})$ is not even monotone.
\end{enumerate}

\medbreak
The next result extends Attouch's theorem (see \cite{Att79,ABT94}) asserting that
\textit{the class of maximal monotone operators on a finite dimensional space
is closed with respect to Painlev\'e-Kuratowski convergence.}

\begin{theo}\label{theo2}
Let $X$ be reflexive strictly convex with a strictly convex dual $X^*$ and
let $\{T_n: X \tos X^*\}$ be a sequence of maximal monotone operators.
Then, when nonempty, the Mosco limit $\mbox{M-}\lim T_n$ is maximal monotone.
\end{theo}
\begin{proof}
Assume that $T=s \times s\mbox{-}\liminf T_n=w \times w\mbox{-}\limsup T_n$,
with $\Dom T\ne \emptyset$.
We show that $T^0=T$.
Invoking again Lemma \ref{lemma1}, we know that, for any $(x,x^*)\in X\times X^*$,
the sequence $\{x_n\}$ of solutions of (\ref{reussi1})
is bounded and any subsequence $\{x_{n_k}\}$ such that $x_{n_k} \rightharpoonup \overline{x}$ and
$J(x_{n_k}-x)\rightharpoonup \eta^*$ satisfies (\ref{reussi2}).
Since
$$
\bigl(x_{n_k}, x^* - J(x_{n_k}-x)\bigr) \in  T_{n_k},\quad \forall k\in\N,
$$
we have $(\overline{x}, x^* - \eta^*) \in w \times w\mbox{-}\limsup T_n$,
hence, $(\overline{x}, x^* - \eta^*) \in T$.
Using (\ref{reussi2}) with $(y, y^*) = (\overline{x}, x^* - \eta^*)$, we get
\begin{equation}\label{reussi2b}
\langle \eta^*, \overline{x} - x\rangle \geq \limsup \|x_{n_k} - x\|^2.
\end{equation}
Now, let $(x,x^*)\in T^0$. We have
$$
\langle \eta^*, x - \overline{x}\rangle = \bigl\langle x^* - (x^* - \eta^*), x - \overline{x} \bigr\rangle \geq 0,
$$
so (\ref{reussi2b}) yields
$$
0 \ge \limsup\|x_{n_k} - x\|^2.
$$
As in the proof of Assertion (1) in Theorem \ref{theo2}, we conclude that $(x,x^*) \in T$, as required.
\end{proof}

\section{Application 1: the variational sum}
\setcounter{equation}{0}
Again in this section, we are given a reflexive Banach space $X$, with strictly convex norms in $X$ and $X^*$,
so that the duality mapping $J:X\to X^*$ is single-valued, bijective and maximal monotone.

Let $T_1, T_2:X \tos X^*$ be maximal monotone.
It is well known that the point-wise Minkowski sum $T_1+T_2$ is a monotone operator,
which is not maximal in general, not even representable.
This is the reason why in recent years different ways of summing two
maximal monotone operators were considered in order to have more chances to get maximality
or at least representability.
We refer to Revalski \cite{Rev10} for a survey on generalized sums and compositions.

One such concept is the so-called \textit{variational sum}, introduced by Attouch-Baillon-Th\'era
\cite{ABT94} in Hilbert spaces, then generalized to reflexive Banach spaces by Revalski-Th\'era
\cite{RT99,RT99b}. It is defined as follows:
$$
T_1 \var T_2 := \bigcap_\mathcal{I}\liminf_n\,(T_{1, \lambda_n}+ T_{2,\mu_n}),
$$
where
$
\mathcal{I}=\Bigl\{ \{(\lambda_n, \mu_n)\}\subset \R^2: \lambda_n,\, \mu_n \geq 0, \,
\lambda_n+ \mu_n >0, \lambda_n, \mu_n \to 0 \Bigr\},
$
and $T_{1, \lambda_n}$, $T_{2,\mu_n}$ are the Yosida approximations of $T_1$ and $T_2$ respectively.
[The original definition uses topological non-sequential limit, but
it is easily seen that both definitions are equivalent.]
\ps
In this setting, the variational sum turns out to be bigger than the point-wise sum:
\begin{equation}\label{garcia}
T_1+T_2\subset T_1 \var T_2.
\end{equation}
This fact was recently established by Garc\'\i a \cite[Corollary 3.7]{Gar09}.
Also, it has been shown that the variational sum of two subdifferentials of lower semicontinuous convex
functions is the subdifferential of the sum of the two functions, hence a maximal monotone operator
(see \cite{ABT94,RT99b}). 
But the general question whether the variational sum of two arbitrary maximal monotone operators
is maximal monotone is still open.
However, we have the following:

\begin{theo}\label{varsum}
Let $X$ be reflexive strictly convex with a strictly convex dual $X^*$ and
let $T_1, T_2:X \tos X^*$ be maximal monotone.
Assume $\Dom (T_1 + T_2) \ne\emptyset$. Then, the variational sum
$T_1 \var T_2$ is a representable extension of the point-wise sum $T_1+T_2$.
\end{theo}
\begin{proof}
For every $\{(\lambda_n, \mu_n)\}\in \mathcal{I}$, the operators $T_n=T_{1, \lambda_n}+ T_{2,\mu_n}$,
$n\in\N$, are maximal monotone. Indeed, for every $n\in\N$, at least one of the parameters
$\lambda_n$ or $\mu_n$ is different from $0$, so at least one of the operators
$T_{1, \lambda_n}$ or $T_{2,\mu_n}$ is single-valued and maximal monotone,
hence their sum is maximal monotone.
On the other hand, by (\ref{garcia}),
$$
T_1+ T_2 \subset T_1 \var T_2 \subset \liminf\, (T_{1, \lambda_n} + T_{2, \mu_n}),
$$
hence $\Dom (\liminf T_n)\ne\emptyset$.
Therefore, according to Theorem \ref{theo1},
for every $\{(\lambda_n, \mu_n)\}\in \mathcal{I}$, the operator
$$
\liminf\,(T_{1, \lambda_n}+ T_{2,\mu_n})
$$
is representable.
Since the intersection of representable operators is a representable operator, we conclude
that $T_1 \var T_2$ is representable.
\end{proof}


\noindent
{\it Example.}
(See Garc\'\i a-Lassonde-Revalski \cite[Example 3.11]{GLR06} and Garc\'\i a \cite[Example 3.13]{Gar09}.)
Let $X=\ell_2\times \ell_2$ and identify $X^*$ with $X$. Let $\Dom T:=D\times
D$ with
$$D:= \{\{x_n\} \subset \ell_2: \{2^nx_n\} \in \ell_2\},
$$
 and let $T: \Dom T\to X$ be defined by $$T(\{x_n\},
\{y_n\}):=(\{2^ny_n\}, -\{2^nx_n\}).$$ Consider $T_1:=T$ and $T_2:=-T$.
These operators are linear, anti-symmetric and maximal monotone with
common dense domain $\Dom T$.
Clearly, \textit{$T_1+T_2$ is not representable}, because $T_1+T_2$ is not closed
in $X\times X^*$, since $T_1+T_2\equiv 0$ with $\Dom(T_1+T_2)=\Dom T$ proper dense subset of $X$,
while \textit{$T_1 \var T_2$ is representable} (in fact maximal monotone),
since $(T_1 \var T_2)\equiv 0$ with $\Dom (T_1 \var T_2)=X$.

\section{Application 2: the variational composition}
\setcounter{equation}{0}
Let $X$ and $Y$ be reflexive Banach spaces, supplied
with strictly convex norms as well as their dual spaces,
let $T:X \rightrightarrows X^*$ be maximal monotone and let $A: Y \to X$ be
linear continuous with adjoint $A^*:X^*\to Y^*$.
The usual point-wise composition $A^*TA: Y \rightrightarrows Y^*$ is given by
$$
A^*TAy= \bigl\{y^* = A^*x^*:  x^* \in TAx\bigr\}.
$$
As in the case of the point-wise sum, this operator is monotone but in general
not even representable, whence the idea to consider different types of compositions.
Based on the same idea as the variational sum,
the so-called \textit{variational composition}, introduced by Pennanen-Revalski-Th\'era
\cite{PRT03}, is defined by:
$$
(A^*TA)_v = \bigcap_\mathcal{J}\liminf A^*T_{\lambda_n} A,
$$
where
$
\mathcal{J}=\Bigl\{ \{\lambda_n\}\subset \R: \lambda_n>0, \, \lambda_n\searrow  0 \Bigr\},
$
and $T_{\lambda_n}$ is the Yosida approximation of $T$.
[Again the original definition uses topological non-sequential limit, but
it is easily seen that both definitions are equivalent.]
\ps

It is well known that the point-wise composition can be expressed as a point-wise sum of operators.
More specifically, let $N_A, \widetilde{T} : Y \times X \rightrightarrows Y \times X$ be defined by
$$
\widetilde{T}(y,x):= \{0\} \times Tx, \quad \forall (y,x)\in Y \times X,\qquad N_A := \partial \delta_A,
$$
where $\delta_A$ is the indicator function of the graph of $A$, therefore,
$$
N_A(y,x)= \left\{
\begin{array}{ll}
 \bigl\{(A^*x^*, -x^*): x^* \in X^* \bigr\} & \mbox{ if } (y,x) \in A,\smallskip\\
\emptyset & \mbox{ otherwise. }
\end{array}
\right.
$$
Then
\begin{equation}\label{one0}
y^* \in A^*TAy \Leftrightarrow  (y^*, 0) \in (\widetilde{T}+N_A)(y,Ay).
\end{equation}
\ps
Similarly, it turns out that the variational composition can be expressed via an asymmetric variant of the variational
sum. Given maximal monotone operators $T_1, T_2:X \tos X^*$,
consider the following \textit{left variational sum}:
$$
l\sep(T_1 \var T_2):=\bigcap_\mathcal{J}\liminf\,(T_{1, \lambda_n}+ T_{2}).
$$
Clearly, this operator contains the variational sum,
\begin{equation}\label{one00}
T_1 \var T_2 \subset l\sep(T_1 \var T_2),
\end{equation}
and in general the inclusion is proper, as the following example shows.
\ps
\noindent
{\it Example.} Let $X = \R$, $T_1 = \partial\delta_{\{-1\}}$ and $T_2 = \partial\delta_{\{1\}}$.
We have $\Dom(T_1) \cap \Dom(T_2) = \emptyset$, and one verifies that
the usual and the variational sum are the trivial empty operator.
On the other hand, for $\ld > 0$ we have $T_{1, \lambda}(x) = (x + 1)/\ld$
for every $x\in \R$, and it can be checked that
$l\sep(T_1 \var T_2)=\{1\}\times \R$.
\ps
Here is the promised representation of the variational composition as a left variational sum:

\begin{prop}\label{prop4}
Let $X$ and $Y$ be Banach spaces.
Let  $T: X \rightrightarrows X^*$ be maximal monotone and let $A: Y \to X$ be linear
continuous. We have the following equivalence:
\begin{equation}\label{one}
y^* \in (A^*TA)_v(y) \Leftrightarrow (y^*, 0) \in l\sep(\widetilde{T}\var N_A)(y, Ay).
\end{equation}
\end{prop}
\begin{proof}
Notice that for every $\ld>0$ and $(y,x)\in Y\times X$
we have $\widetilde{T}_{\lambda}(y,x)=\{0\}\times T_{\lambda}(x)$.
\ps
Let $(y, y^*) \in (A^*TA)_v$.  By definition,
for any $\{\lambda_n\} \in \mathcal{J}$, there exists
$\{y_n\} \subset X$ such that $(y_n, A^*T_{\lambda_n}Ay_n) \to (y, y^*)$.
Observe that
$
(A^*T_{\lambda_n}Ay_n,  - T_{\lambda_n}Ay_n)\in N_A(y_n,Ay_n),
$
hence
$$
(A^*T_{\lambda_n}Ay_n, 0)=  (0, T_{\lambda_n}Ay_n)+(A^*T_{\lambda_n}Ay_n, - T_{\lambda_n}Ay_n)
\in (\widetilde{T}_{\lambda_n}+N_A)(y_n,Ay_n).
$$
Since $(A^*T_{\lambda_n}Ay_n, 0) \to (y^*, 0)$ and $(y_n,Ay_n)\to (y,Ay)$, by definition of $\liminf$
we get
$$((y^*, 0),(y, Ay))\in \liminf\,(\widetilde{T}_{\lambda_n}+N_A).$$
Since this is valid for any $\{\lambda_n\} \in \mathcal{J}$,
we conclude that $(y^*, 0) \in l\sep(\widetilde{T}\var N_A)(y, Ay)$.
\ps
Conversely, let $(y^*, 0) \in l\sep(\widetilde{T}\var N_A)(y, Ay)$. Then,
for any $\{\lambda_n\} \in \mathcal{J}$, there exists
$\{(y_n,x_n)\} \subset Y\times X$ and $\{(y^*_n,x^*_n)\} \subset Y^*\times X^*$ with
$(y^*_n,x^*_n)\in (\widetilde{T}_{\lambda_n}+N_A)(y_n,x_n)$
such that $(y_n,x_n)\to (y,Ay)$ and $(y^*_n,x^*_n)\to (y^*, 0)$. From the definition of $N_A$,
we get $x_n=Ay_n$ and from
$$
(y^*_n,x^*_n)=(0,T_{\lambda_n}Ay_n)+(y^*_n, x^*_n-T_{\lambda_n}Ay_n)\in \widetilde{T}_{\lambda_n}+N_A
$$
we get
$y^*_n=A^*(T_{\lambda_n}Ay_n-x^*_n)=A^*T_{\lambda_n}Ay_n-A^*x^*_n$.
Thus,
$(y_n,y^*_n+A^*x^*_n)\in A^*T_{\lambda_n}A$ and since $(y_n,y^*_n+A^*x^*_n)\to (y,y^*)$, we have
$$(y,y^*)\in \liminf A^*T_{\lambda_n}A.$$
We conclude that $y^* \in (A^*TA)_v(y)$.
\end{proof}

We are now ready to prove the analog of Theorem \ref{varsum}.
\begin{theo}
Let $X$ and $Y$ be reflexive strictly convex with strictly convex duals.
Let  $T: X \rightrightarrows X^*$ be maximal monotone and let $A: Y \to X$ be linear
continuous.
Assume $\Dom A^*TA \ne\emptyset$.
Then, the variational composition $(A^*TA)_v$ is a representable extension
of the point-wise composition $A^*TA$.
\end{theo}
\begin{proof}
First we show that $(A^*TA)_v$ contains $A^*TA$.
Let $(y, y^*) \in A^*TA$. By (\ref{one0}),
$(y^*, 0) \in (\widetilde{T}+N_A)(y,Ay)$.
Observe that $Y\times X$ supplied with the square norm on the product
is a reflexive Banach space with a strictly convex norm as well as its dual.
So we may apply (\ref{garcia})
with the maximal monotone mappings $\widetilde{T}$ and $N_A$ acting on
$Y\times X$ to obtain $(y^*, 0) \in (\widetilde{T}\var N_A)(y,Ay)$,
hence, by (\ref{one00}), $(y^*, 0) \in l\sep(\widetilde{T}\var N_A)(y,Ay)$.
This amounts to $(y, y^*) \in (A^*TA)_v$ according to Proposition \ref{prop4}.
\ps
Next, we show that $(A^*TA)_v$ is representable. This follows in the same way as for the variational sum
(Theorem \ref{varsum}). Indeed, for any $\ld_n > 0$ the operator $A^*T_{\ld_n}A : Y \to Y$ is single-valued
everywhere defined maximal monotone because so is $T_{\ld_n}$ (see, e.g., \cite{Sim08}).
Moreover $\Dom (\liminf A^*T_{\ld_n}A )\ne\emptyset$ because $\Dom A^*TA \ne\emptyset$
and $A^*TA \subset (A^*TA)_v\subset \liminf A^*T_{\ld_n}A$.
Hence, in view of Theorem \ref{theo1},
for every $\{\lambda_n\}\in \mathcal{J}$, the operator
$
\liminf\,A^*T_{\ld_n}A
$
is representable, so also is $(A^*TA)_v$ as intersection of such representable operators.
\end{proof}

\end{document}